\documentclass[12pt,a4paper]{amsart}
\usepackage{graphicx,amssymb}
\usepackage{amsmath,amssymb,amscd}
\input xy
\xyoption{all}
\usepackage[all]{xy}
\usepackage{hyperref}

\newcommand{\ra}{\rightarrow}

\newcommand{\CC}{\mathbb C}

\newcommand{\PP}{\mathbb P}

\newcommand{\cO}{\mathcal{O}}

\newcommand{\Hom}{\mbox{Hom}}

\newcommand{\rk}{\operatorname{rk}}
\newcommand{\Cl}{\operatorname{Cliff}}

\theoremstyle{plain}
\newtheorem{theorem}{Theorem}[section]
\newtheorem{lem}[theorem]{Lemma}
\newtheorem{prop}[theorem]{Proposition}
\newtheorem{cor}[theorem]{Corollary}

\newtheorem{rem}[theorem]{Remark}

\newtheorem{ques}[theorem]{Question}
\numberwithin{equation}{section}

\begin{document}
\title[Bundles of rank three]{On bundles of rank 3 computing Clifford indices}

\author{H. Lange}
\author{P. E. Newstead}

\address{H. Lange\\Department Mathematik\\
              Universit\"at Erlangen-N\"urnberg\\
              Cauerstrasse 11\\
              D-$91058$ Erlangen\\
              Germany}
              \email{lange@mi.uni-erlangen.de}
\address{P.E. Newstead\\Department of Mathematical Sciences\\
              University of Liverpool\\
              Peach Street, Liverpool L69 7ZL, UK}
\email{newstead@liv.ac.uk}

\thanks{Both authors are members of the research group VBAC (Vector Bundles on Algebraic Curves). The second author would like to thank the
Department Mathematik der Universit\"at
         Erlangen-N\"urnberg for its hospitality.}
\maketitle
\begin{center}{\it Dedicated to the memory of Masaki Maruyama}
\end{center}
\begin{abstract}
Let $C$ be a smooth irreducible projective algebraic curve defined over the complex numbers. The notion of the Clifford index of $C$ was extended a few years ago to semistable bundles of any rank. Recent work has been focussed mainly on the rank-2 Clifford index, although interesting results have also been obtained for the case of rank 3. In this paper we extend this work, obtaining improved lower bounds for the rank-3 Clifford index. This allows the first computations of the rank-3 index in non-trivial cases and examples for which the rank-3 index is greater than the rank-2 index.
\end{abstract}

\section{Introduction} \label{intro}

Let $C$ be a smooth irreducible projective algebraic curve defined over the complex numbers. The idea of generalising the classical Clifford index $\Cl(C)$ to higher rank vector bundles was proposed some 20 years ago, but formal definitions and the development of a basic theory took place much more recently \cite{cliff}. Since then, there have been major developments, in particular the construction of curves for which the rank-2 Clifford index $\Cl_2(C)$ is strictly less than $\Cl(C)$  \cite{fo, fo2, cl,cl3,cl1}, thus producing counter-examples to a conjecture of Mercat \cite{m}. A good deal is now known about bundles computing $\Cl_2(C)$ \cite{cl3}.

Examples are also known for $g=9$ and $g\ge11$ for which the rank-3 Clifford index $\Cl_3(C)$ is strictly smaller than $\Cl(C)$ \cite{lmn,fo2} and lower bounds for $\Cl_3(C)$ had previously been established in \cite{cl2}.   However, with the exception of the case where $\Cl(C)\le2$ (when $\Cl_3(C)=\Cl(C)$ \cite[Proposition 3.5]{cliff}), no actual values of $\Cl_3(C)$ are known. In the present paper, we improve the lower bounds of \cite{cl2} in various circumstances. As a result, we are able to compute values of $\Cl_3(C)$ in some cases and to give examples for which $\Cl_3(C)>\Cl_2(C)$, thus answering in the affirmative Question 5.7 in \cite{lmn}.

Following definitions and some preliminary results in Section \ref{prelim}, we consider in Section \ref{minimal} the curves of minimal rank-2 Clifford index constructed in \cite{cl1}; these are good candidates for having $\Cl_3(C)>\Cl_2(C)$ and we prove in particular

\medskip
\noindent{\bf  Theorem \ref{thm2.9}.}\begin{em}
If $16 \leq g \leq 24$, then there exists a curve $C$ of genus $g$ such that
$$
\Cl_3(C) > \Cl_2(C).
$$\end{em}This could hold also for other values of $g$ (see Theorem \ref{thm2.5} and Remark \ref{rem2.10}). 

In Section \ref{improved}, we establish the following improved lower bound for $\Cl_3(C)$ when $\Cl_2(C)=\Cl(C)$.

\medskip
\noindent{\bf Theorem \ref{thm3.1}}\begin{em}
Let $C$ be a curve of genus $g \geq 7$ such that $\Cl_2(C) = \Cl(C) \geq 2$. Then 
$$
\Cl_3(C)  \geq \min \left\{ \frac{d_9}{3} -2, \frac{2\Cl(C) +2}{3} \right\}.
$$Moreover, if $\Cl_3(C)<\Cl(C)$, then any bundle computing $\Cl_3(C)$ is stable.\end{em}

\noindent(For the definition of the gonalities $d_r$, see Section \ref{prelim}.) These new bounds may appear to be a minor improvement on those of \cite{cl2}, but they are in some sense best possible in the light of current knowledge and have surprisingly strong consequences. In particular, in the course of proving Theorem \ref{thm3.1}, we are able to show that $\Cl_3(C)=\frac{10}3$ for the general curve of genus $9$ (Proposition \ref{prop3.8} and Corollary \ref{corg=9}); to our knowledge, this is the first complete computation of $\Cl_3(C)$ for any curve with $\Cl(C)>2$.

Section \ref{plane} is concerned with the case of plane curves, especially smooth plane curves. We note first that, if $C$ is a smooth plane curve of degree $\delta\ge6$, Theorem \ref{thm3.1} implies that $\Cl_3(C)\ge\frac{2\delta-6}3$ (Proposition \ref{prop4.1}). The main result of this section identifies all possible bundles for which this lower bound could be attained.

\medskip
\noindent{\bf Theorem \ref{thm5.6}}\begin{em}
If $C$ is a smooth plane curve of degree $\delta \geq 7$ and $\Cl_3(C) = \frac{2 \delta -6}{3}$,
then any bundle $E$ computing $\Cl_3(C)$ is stable and fits into an exact sequence 
\begin{equation*}
0 \ra E_H \ra E \ra H \ra 0
\end{equation*}
and all sections of $H$ lift to $E$. Moreover, such extensions exist if and only if $h^0(E_H \otimes E_H) \geq 10$.\end{em}

\noindent(Here $H$ denotes the hyperplane bundle on $C$ and $E_H$ is defined by the evaluation sequence $0\to E_H^*\to H^0(E)\otimes{\mathcal O}_C\to E\to 0$.) For the normalisation of a nodal plane curve, we prove a similar but more complicated result (Theorem \ref{thmnodal}).

In Section \ref{three} we study curves with $\Cl_3(C)=3$. Our main result here is 

\medskip
\noindent{\bf Theorem \ref{thm6.8}}\begin{em} Let $C$ be a curve of genus $g \geq 9$ with $\Cl(C) =3$. If $d_2 > 7$, and in particular if $g \geq 16$, then
$$
\Cl_3(C) = 3.
$$
For all $g \geq 9$ there exist curves with these properties.\end{em}

\medskip
\noindent For curves with $\Cl_3(C)=3$ and $d_2=7$ (which can exist only for $7\le g\le15$) or with $g=8$ and $d_2=8$, we have $\frac83\le \Cl_3\le3$, but we do not know the precise value of $\Cl_3(C)$. We do however give a list of all bundles which could compute $\Cl_3(C)$ if $\Cl_3(C)=\frac83$ (Propositions \ref{prop5.5}, \ref{prop5.6} and \ref{prop5.7}). The problem is therefore reduced to that of determining whether any of these bundles exists.

In Section \ref{cs}, we prove that, if $\Cl_3(C)\le\Cl_2(C)$ and $E$ computes $\Cl_3(C)$, then the coherent system $(E,H^0(E))$ is $\alpha$-semistable for all $\alpha>0$; if in addition $E$ is stable, then $(E,H^0(E))$ is $\alpha$-stable for all $\alpha>0$. (In fact we prove a result for rank $n$ (Proposition \ref{prop7.1}) of which this is the case $n=3$.) These results are of interest in connection with a conjecture of D. C. Butler. 

Finally, Section \ref{further} contains further comments and a discussion of open problems.

We suppose throughout that $C$ is a smooth irreducible projective algebraic curve defined over $\CC$ and denote the canonical bundle on $C$ by $K_C$. For a vector bundle $F$ on $C$, we denote the degree of $F$ by $d_F$ and its slope by $\mu(F):=\frac{d_F}{\rk F}$.

\section{Definitions and preliminaries}\label{prelim}

\medskip
We recall first the definition of $\Cl_n(C)$. For any vector bundle $E$ of rank $n$ and degree $d$ on $C$, we define
$$
\gamma(E) := \frac{1}{n} \left(d - 2(h^0(E) -n)\right) = \mu(E) -2\frac{h^0(E)}{n} + 2.
$$
If $C$ has genus $g \geq 4$, we then define, for any positive integer $n$,
$$
\Cl_n(C):= \min_{E} \left\{ \gamma(E) \;\left| 
\begin{array}{c} E \;\mbox{semistable of rank}\; n, \\
h^0(E) \geq 2n,\; \mu(E) \leq g-1
\end{array} \right. \right\}
$$
(this invariant is denoted in \cite{cliff, cl2, cl4, cl, cl3} by $\gamma_n'$). Note that $\Cl_1(C) = \Cl(C)$ is the usual Clifford index of the curve $C$. 
We say that $E$ {\it contributes to} $\Cl_n(C)$ if $E$ is semistable of rank $n$ with $h^0(E) \geq 2n$ and $\mu(E) \leq g-1$. If in addition $\gamma(E) = \Cl_n(C)$, we say that $E$ {\it computes} $\Cl_n(C)$. 
Moreover, as observed in \cite[Proposition 3.3 and Conjecture 9.3]{cliff}, the conjecture of \cite{m} can be restated in a slightly weaker form as

\medskip
\noindent{\bf Conjecture.} $\Cl_n(C)=\Cl(C)$. 

\medskip
In fact, for $n=2$, this form of the conjecture is equivalent to the original (see \cite[Proposition 2.7]{cl3}).

\begin{lem} \label{l2.1}
The conjecture is valid in the following cases
\begin{enumerate}
\item[(i)] $\Cl(C) \leq 2$,
\item[(ii)]$ n= 2 \; \mbox{and} \; \Cl(C) \leq 4$.
\end{enumerate}
\end{lem}   

\begin{proof}
See \cite[Propositions 3.5 and 3.8]{cliff}.
\end{proof}

However, the conjecture is known to fail in many other cases (see  \cite{fo, fo2, cl,cl3,cl1}).
For $n = 3$ it fails for the general curve of genus 9 or 11 (see \cite{lmn}) and for curves of genus $\ge12$ contained in K3 surfaces \cite[Corollary 1.6]{fo2}. For $n=2$ it is still conjectured to hold for the general curve of any genus (see \cite{fo}).
Note that in any case
\begin{equation} \label{e2.1}
\Cl_n(C) \leq \Cl(C)
\end{equation}
(see \cite[Lemma 2.2]{cliff}) and for $n = 2$ we have the lower bound
\begin{equation} \label{e2.2}
\Cl_2(C) \geq \min \left\{ \Cl(C), \frac{\Cl(C)}{2} + 2 \right\}
\end{equation}
(see \cite[Proposition 3.8]{cliff}).

The {\em gonality sequence} 
$d_1,d_2,\ldots,d_r,\ldots$ of $C$ is defined by 
$$
d_r := \min \{ d_L \;|\; L \; \mbox{a line bundle on} \; C \; \mbox{with} \; h^0(L) \geq r +1\}.
$$
We have always $d_r<d_{r+1}$ and $d_{r+s}\le d_r+d_s$; in particular $d_n\le nd_1$ for all $n$ (see \cite[Section 4]{cliff}). 
We say that $d_r$ {\em computes} $\Cl(C)$ if $d_r \leq g-1$ and $d_r-2r=\Cl(C)$ and that $C$ has {\em Clifford dimension} $r$ if $r$ is 
the smallest integer for which $d_r$ computes $\Cl(C)$. Note also \cite[Lemma 4.6]{cliff}
\begin{equation}\label{dr}
d_r\ge\min\{\Cl(C)+2r,g+r-1\}.
\end{equation}

We recall that $\Cl(C)\le\left[\frac{g-1}2\right]$ with equality on the general curve of genus $g$. In fact equality holds on any \emph{Petri curve}, that is any curve for which the multiplication map
$$H^0(L)\otimes H^0(K_C\otimes L^*)\to H^0(K_C)$$
is injective for every line bundle $L$ on $C$. Moreover
\begin{equation}\label{dr2}
d_r\le g+r-\left[\frac{g}{r+1}\right],
\end{equation}
again with equality on any Petri curve.

In the following sections, we shall need a few basic results. The first is the lemma of Paranjape and Ramanan \cite[Lemma 3.9]{pr}, which can be stated as follows.
\begin{lem}\label{pr}
Let $E$ be a bundle of rank $n$ and degree $d$ on $C$ with $h^0(E)=n+s$ possessing no proper subbundle $F$ with $h^0(F)>\rk F$. Then $d\ge d_{ns}$.
\end{lem}

As a complement to this lemma in the case $n=2$, we have (see \cite[Lemma 2.6]{cl3})
\begin{lem}\label{lempr2}
Suppose that $F$ is a semistable bundle of rank $2$ and degree $\le2g-2$ which possesses a subbundle $M$ with $h^0(M)\ge2$. Then $\gamma(F)\ge \Cl(C)$, with equality if and only if $\gamma(M)=\gamma(F/M)=\Cl(C)$ and all sections of $F/M$ lift to $F$.
\end{lem}

\begin{prop}\label{newprop}
Suppose that either $\Cl_3(C)<\Cl_2(C)=\Cl(C)$ or $\Cl_3(C)\le\Cl_2(C)<\Cl(C)$ and let $E$ be a bundle computing $\Cl_3(C)$. Then $E$ is stable.
\end{prop}

\begin{proof} Suppose that $E$ is strictly semistable. Then $E$ is S-equivalent to a bundle of the form $F\oplus L$, where $\rk F=2$, $\rk L=1$ and both bundles have the same slope as $E$. Moreover $\gamma(E)\ge\gamma(F\oplus L)$. 

Note that either $F$ contributes to $\Cl_2(C)$ or $L$ contributes to $\Cl(C)$. If both of these hold, then clearly $\gamma(F\oplus L)\ge\frac{2\Cl_2(C)+\Cl(C)}3$. If $F$ does not contribute to $\Cl_2(C)$, then $h^0(F)\le3$, so 
$$\gamma(F)\ge \mu(F)-1=d_L-1>\gamma(L).$$ 
Since $\gamma(L)\ge\Cl(C)$, it follows that $\gamma(F\oplus L)>\Cl(C)$. Finally, suppose $L$ does not contribute to $\Cl(C)$. Then $$\gamma(L)\ge d_L=\mu(F)>\gamma(F)\ge\Cl_2(C),$$ so $\gamma(F\oplus L)>\Cl_2(C)$. In all cases, we obtain the contradiction $\gamma(E)>\Cl_3(C)$.
\end{proof}

For the next result recall that, if $L$ is a generated line bundle with $h^0(L) = 1 +u$, then the evaluation sequence
\begin{equation} \label{e2.5}
0 \ra E_L^* \ra H^0(L) \otimes \cO_C \ra L \ra 0
\end{equation} 
defines a vector bundle $E_L$ of rank $u$ and degree $d_L$.

\begin{lem} \label{l2.4}
If $u =2$ and $d_L = d_2$ in \eqref{e2.5}, then $E_L$ is semistable. Moreover, if $d_2 < 2d_1$, then $E_L$ is stable and $h^0(E_L) = 3$.
\end{lem}

\begin{proof}
See \cite[Proposition 4.9 and Theorem 4.15]{cliff}.
\end{proof}

\begin{prop} \label{p2.5}
Suppose that $3\Cl(C) \geq 2d_2 - 6$ and $\Cl_2(C) = \Cl(C)$. Let $F$ be a stable bundle of rank $2$ and degree $d_2$ with $h^0(F) =3$ and let $L$ be a line bundle of degree $d_2$ with $h^0(L) =3$. Suppose further that
\begin{equation} \label{eq2.3}
0 \ra F \ra E \ra L \ra 0
\end{equation}
is a non-trivial extension with $h^0(E) =6$.
Then $E$ is semistable and generated. Moreover, extensions  \eqref{eq2.3} with these properties exist if and only if $h^0(F \otimes E_L) \geq 10$.
\end{prop}

\begin{proof}
If $F$ is not generated, then it possesses a subsheaf $F'$ of rank 2 and degree $d_2-1$ such that $h^0(F') = 3$. 
Moreover, $F'$ is semistable. This contradicts \cite[Proposition 4.12]{cliff}. 

Since also $h^0(F^*) =0$, we have an exact sequence
$$
0 \ra F^* \ra H^0(F)^* \otimes \cO_C \ra M \ra 0
$$
where $M \simeq \det F$ has degree $d_2$ and $h^0(M) \geq 3$. Hence $h^0(M) = 3$ and $F \simeq E_M$. The semistability of $E$ now follows as in \cite[Proposition 3.5]{lmn} noting that the inequality $3d_1 \geq 2d_2$ is weaker than $3\Cl(C) \geq 2d_2 -6$. Moreover, $E$ is obviously generated.

For the last assertion note that a non-trivial extension \eqref{eq2.3} with $h^0(E) =6$ corresponds to a non-zero element of the kernel of the natural map
$$
H^1(L^* \otimes F) \ra \Hom(H^0(L), H^1(F)) = H^0(L)^* \otimes H^1(F).
$$
Now consider the sequence
$$
0 \ra L^* \otimes F \ra H^0(L)^* \otimes F \ra E_L \otimes F \ra 0.
$$
Since $F$ is stable, $H^0(L^* \otimes F) = 0$. So we have an exact sequence
$$
0 \ra H^0(L)^* \otimes H^0(F) \ra H^0(E_L \otimes F) \ra H^1(L^* \otimes F) \ra H^0(L)^* \otimes H^1(F).
$$
Hence there exists a non-trivial extension \eqref{eq2.3} with $h^0(E) = 6$ if and only if 
$h^0(E_L \otimes F) > h^0(L)\cdot h^0(F) = 9$.
\end{proof}

\section{Curves with minimal rank-2 Clifford index}\label{minimal}

In this section, we let $C$ be a curve of genus $g \geq 11$ with 
\begin{equation}\label{eq:new1}
\Cl(C) = \left[ \frac{g-1}{2} \right] \quad  \mbox{and} \quad 
\Cl_2(C) = \frac{1}{2}\left[ \frac{g-1}{2} \right] + 2.
\end{equation}
Such curves exist by \cite{cl1} and \cite{fo}. Note that \eqref{eq:new1} implies that $\Cl_2(C)<\Cl(C)$.  
By  \eqref{e2.2}, $\Cl_2(C)$ takes its minimum value for the given value of $\Cl(C)$, so these curves are good candidates for obtaining values of $\Cl_3(C)$ greater than $\Cl_2(C)$. A further implication of \eqref{eq:new1} is that $d_4\le\Cl(C)+8$ (see \cite[Theorem 5.2]{cliff}). On the other hand, $d_4\ge\Cl(C)+8$  for any curve of genus $\ge8$ by \eqref{dr}, so, for our curves, we have $d_4=\Cl(C)+8$. This implies that $C$ cannot be a Petri curve for $g\ge12$. On the other hand, it is known that, for $g=11$, $C$ can be Petri \cite[Theorem 1.5]{mlc}.

We follow the arguments of \cite{cl2}.
\begin{prop}  \label{prop2.1}
Let $E$ be a semistable bundle of degree $d$ computing $\Cl_3(C)$. If $g \geq 19$ and 
$d < 2g -2 + \frac{1}{2} \left[\frac{g-1}{2} \right]$, then either
$$
\gamma(E) > \Cl_2(C) \quad   \mbox{or} \quad \gamma(E) \geq \frac{d_9}{3} -2.
$$
Moreover, if $\gamma(E) \leq \Cl_2(C)$, then $E$ has no proper subbundle $F$ with $h^0(F) \geq \rk F +1$.
\end{prop}

\begin{proof}
By \cite[Proposition 2.4]{cl2} we have
\begin{eqnarray*}
\gamma(E) \geq \min \left\{\frac{d_9}{3} -2,\right.&& \Cl_2(C),\quad \frac{2 \Cl(C) + 1}{3},\\& & \left.\frac{1}{3}(2 \Cl(C) + 2g -d +4) \right\}. 
\end{eqnarray*}
Note that the bound $\Cl_2(C)$ enters only in \cite[formula (2.2)]{cl2} and is a strict inequality. Moreover, the condition on $d$ is necessary and sufficient for
$$
\frac{1}{3}(2 \Cl(C) + 2g -d + 4) > \Cl_2(C).
$$
If $g \geq 23$, then 
$$
\frac{2 \Cl(C) + 1}{3} > \Cl_2(C)
$$ 
and we are finished. For $g < 23$ we need to improve the bound $\frac{2 \Cl(C) + 1}{3}$. 
The points where this enters in the proof of \cite[Proposition 2.4]{cl2} are \cite[Lemma 2.2(i)]{cl2} and \cite[formula (2.3)]{cl2}. (The inequality at the end of the proof of \cite[Lemma 2.2]{cl2} can be replaced by $\gamma(E) \geq \frac{4 \Cl(C) + 2}{3}$ which is clearly greater than $\Cl_2(C)$.) 

For \cite[Lemma 2.2(i)]{cl2}, we have
$$ 
\gamma(E)\ge\frac13(\Cl(C)+d_6)-2.
$$
By \eqref{dr},  
$$
d_{6} \geq \min \left\{ \left[ \frac{g-1}{2} \right] + 12,g+5 \right\}=\left[\frac{g-1}2\right]+12
$$
for $g\ge12$. So
$$\gamma(E)\ge\frac23\left[\frac{g-1}2\right]+2>\Cl_2(C).$$

For \cite[formula (2.3)]{cl2}, the estimate enters in 2 different cases. In the first case we have, for some integer $t\ge1$,
\begin{equation}\label{eq:new2}
\gamma(E) \geq \frac{2 \Cl(C) + 2t}{3} \geq \frac{2 \Cl(C) + 2}{3} > \Cl_2(C)
\end{equation}
for $g \geq 19$. In the second case we have
\begin{eqnarray*}
\gamma(E) & \geq & \frac{2t+4g-4}{3} - \frac{d}{3}\\
& > & \frac{2t + 4g -4}{3} -  \frac{2g -2 + \frac{1}{2} \left[ \frac{g-1}{2} \right]}{3}\\
& = &  \frac{2t+2g -2 - \frac{1}{2} \left[ \frac{g-1}{2} \right]}{3}  > \Cl_2(C).
\end{eqnarray*}
\end{proof}

\begin{rem}\label{rem18}
{\em
For $g \leq 18$ one can have $\gamma(E)\le\Cl_2(C)$ in \eqref{eq:new2}. Since $g\ge11$, this can occur only if $t=1$. If $15\le g\le18$, all the other inequalities in the proof of  \eqref{eq:new2} must be equalities. In particular, $d_2=d_{2t}$ computes $\Cl(C)$. For $16\le g\le18$, one can check that all the hypotheses of \cite[Theorem 9.1]{mlc} hold except that we do not know whether the quadratic form in the statement of that theorem can take the value $-1$. However this does not matter in view of \cite[Corollaries 2.4 and 2.6]{cl1}. So \cite[Theorem 9.2]{mlc} applies and a simple calculation shows that this gives $d_2\ge\Cl(C)+5$, contradicting the assumption that $d_2$ computes $\Cl(C)$. It follows that, for $16\le g\le18$, there exists a curve $C$ satisfying \eqref{eq:new1} for which the conclusion of Proposition \ref{prop2.1} holds.}
\end{rem}

\begin{prop} \label{prop2.2}
Let $E$ be a semistable bundle of degree $d$ computing $\Cl_3(C)$. Suppose that $E$ possesses a proper subbundle $F$ of maximal 
slope with $\rk F =2$. If 
\begin{equation} \label{eqn2.1}
d>\max \left\{ \frac{3}{4} \left[ \frac{g-1}{2} \right] + g + 12, \frac{9}{2} \left[ \frac{g-1}{2} \right] - 2g +30 \right\} 
\end{equation}
and
\begin{equation}\label{eqnextra}
d < 4g - \frac{3}{2} \left[ \frac{g-1}{2} \right] - 12,
\end{equation} then 
$$
\gamma(E) > \Cl_2(C).
$$
\end{prop}

\begin{proof}
We use the bound of \cite[Lemma 3.2]{cl2}. It is clear that 
$$
\frac{\Cl(C) + 2\Cl_2(C)}{3} > \Cl_2(C).
$$
Moreover, by simple computations,
$$
\frac{\Cl(C)}{3} + \frac{2d-2g-6}{9} > \Cl_2(C)   \Leftrightarrow d > \frac{3}{4} \left[ \frac{g-1}{2}  \right] +g + 12,
$$   
$$
\frac{2\Cl_2(C)}{3} + \frac{d}{9} > \frac{2\Cl_2(C)}{3} + \frac{1}{12} \left[ \frac{g-1}{2} \right] + \frac{g+12}{9} > \Cl_2(C),
$$
$$
\frac{2\Cl_2(C)}{3} + \frac{4g-d-6}{9} > \Cl_2(C) \Leftrightarrow d < 4g - \frac{3}{2} \left[ \frac{g-1}{2} \right] -12,
$$ 
$$
\frac{d+2g-12}{9} > \Cl_2(C) \Leftrightarrow d > \frac{9}{2} \left[ \frac{g-1}{2} \right] -2g + 30.
$$
\end{proof}

\begin{rem} \label{rem2.3}
{\em 
If $g=32$ or $g \geq 34$, we have 
$$
4g - \frac{3}{2} \left[ \frac{g-1}{2} \right] - 12 > 3g-3.
$$
Since we have always $d \leq 3g-3$, this means we can delete the inequality \eqref{eqnextra} in this case.

If $g < 34$ and $g\ne32$, we can have 
\begin{equation}
4g - \frac{3}{2} \left[ \frac{g-1}{2} \right] -12 \leq d \leq 3g-3 
\end{equation}
in which case 
$$
\frac{2 \Cl_2(C)}{3} + \frac{4g-d-6}{9} \leq \Cl_2(C).
$$
This allows the possibility of a bundle $E$ computing $\Cl_3(C)$ with $\gamma(E) \leq \Cl_2(C)$ and sitting in an exact sequence
$$
0 \ra F \ra E \ra E/F \ra 0
$$
with $F$ of maximal slope and rank $2$ and $h^0(F) \geq 4, \; h^1(E/F) \leq 1$ and $d_{E/F} > g-1$. In the next proposition, we show that this still implies that $\gamma(E)>\Cl_2(C)$ if $g\ge16$.
}
\end{rem}
\begin{prop}  \label{prop2.8}
Let $E$ be a semistable bundle  of degree $d$ computing $\Cl_3(C)$. Suppose that $E$ possesses a proper subbundle of maximal slope with $\rk F =2$, $h^0(F)  \geq 4, \;h^1(E/F) \leq 1$ and $d_{E/F} >g-1$. If $g \geq 16$, then 
$$
\gamma(E) > \Cl_2(C).
$$
\end{prop}

\begin{proof}
If $h^1(E/F) = 0$, then, using \cite{ms}, we obtain
$$
\gamma(E/F)=\gamma((E/F)^*\otimes K_C)=2g-d_{E/F} \geq \frac{4g-d}{3}.
$$
Since $\gamma(F) \geq \Cl_2(C)$, this means that $\gamma(E) > \Cl_2(C)$ provided $4g-d > 3 \Cl_2(C)$. A simple computation (using $d\le3g-3$) shows that this holds for $g \geq 10.$ So we can suppose $h^1(E/F) = 1$.

Suppose $\gamma(E) \leq \Cl_2(C)$. As in the proof of \cite[Lemma 3.2]{cl2} we get $d_{E/F} \leq \frac{2g+d}{3}$ or equivalently
$$
\gamma(E/F) \geq \frac{4g-d-6}{3}.
$$
Now
\begin{equation}  \label{eqn2.3}
\frac{2 \gamma(F) +  \gamma(E/F)}{3} \leq \gamma(E) \leq \frac{1}{2} \left[ \frac{g-1}{2} \right]  +2.
\end{equation}
So 
$$
2\gamma(F) \leq \frac{3}{2} \left[ \frac{g-1}{2} \right]  + 6 - \gamma(E/F)
\leq \frac{3}{2} \left[ \frac{g-1}{2} \right]  +6 -  \frac{4g-d-6}{3}.
$$
Hence
$$
h^0(F) = 2 - \gamma(F) + \frac{d_F}{2} \geq - \frac{3}{4} \left[ \frac{g-1}{2} \right]  -2 + \frac{2g+d}{6}.
$$
If $F$ possesses a line subbundle with $h^0 \geq 2$, then by Lemma \ref{lempr2}, $\gamma(F) \geq \Cl(C)$ which contradicts \eqref{eqn2.3}. So by Lemma \ref{pr},
$$
d_F \geq d_t \quad \mbox{with} \quad t = 2(h^0(F)-2) \geq \frac{2g+d}{3} - \frac{3}{2} \left[ \frac{g-1}{2} \right]  -8.
$$
Since $d_t \geq \min \{\Cl(C) + 2t ,g+t-1\}$ by \eqref{dr}, it suffices to show that
\begin{equation} \label{eqn2.4}
d_F < \frac{5g+d}{3} - \frac{3}{2} \left[ \frac{g-1}{2} \right] - 9 
\end{equation}
and
\begin{equation} \label{eqn2.5}
d_F < \frac{4g+2d}{3} - 2 \left[ \frac{g-1}{2} \right] - 16. 
\end{equation}
Since we are assuming that $\gamma(E) \leq \Cl_2(C)$ and we know that $\gamma(F) \geq \Cl_2(C)$, we must have $\gamma(E/F) \leq \Cl_2(C)$, i.e.
$$
d_{E/F} \geq 2g-4 - \frac{1}{2} \left[ \frac{g-1}{2} \right]
$$
and hence
$$
d_F \leq d - 2g+4 + \frac{1}{2} \left[ \frac{g-1}{2} \right].
$$
So for \eqref{eqn2.4} it is enough to prove that
$$
d - 2g + 4 + \frac{1}{2} \left[ \frac{g-1}{2} \right] < \frac{5g+d}{3} - \frac{3}{2} \left[ \frac{g-1}{2} \right] -9.
$$
Using $d \leq 3g-3$, it is sufficient to show that
$$
-10g + 66 + 12 \left[ \frac{g-1}{2} \right] < 0,
$$
which is valid for $g \geq 16$.

For \eqref{eqn2.5} it is enough to prove that
$$
d - 2g + 4 + \frac{1}{2} \left[ \frac{g-1}{2} \right] < \frac{4g+2d}{3} - 2 \left[ \frac{g-1}{2} \right] -16.
$$
Again using $d \leq 3g-3$, it is sufficient to show that
$$
-14g +114 + 15 \left[ \frac{g-1}{2} \right] < 0,
$$
which is valid for $g \geq 16$.
\end{proof}

\begin{prop} \label{prop2.4}
Let $E$ be a semistable bundle of degree $d$ computing $\Cl_3(C)$. Suppose that $E$ possesses a proper 
subbundle $L$ of maximal slope with $\rk L = 1$. If 
\begin{equation}\label{eqdg} 
d > g + \frac{3}{2} \left[ \frac{g-1}{2} \right] + 6,
\end{equation} 
then
$$
\gamma(E) > \Cl_2(C).
$$
\end{prop}

\begin{proof}
We follow the proof of \cite[Lemma 3.1]{cl2}. Clearly
$$
\frac{\Cl(C) + 2\Cl_2(C)}{3} > \Cl_2(C).
$$
Moreover,
$$
\frac{\Cl(C)}{3} + \frac{2d-6}{9} > \Cl_2(C)
$$
under the assumption on $d$. 

It remains to handle the case where $h^0(L) \leq 1$. In this case we have an exact sequence 
$$
0 \ra L \ra E \ra Q \ra 0
$$
and, by \cite{ms}, 
$$
\mu(Q) - \mu(L) \leq g.
$$
Moreover, every line subbundle $M$ of $Q$ must have $d_M \leq d_L$ (otherwise the pullback of 
$M$ to $E$ would have slope greater than $d_L$). We can assume $M$ has maximal slope as a subbundle of $Q$, so, again by \cite{ms},
$$
\mu(Q/M) - \mu(M) \leq g.
$$
In other words,
$$
d - d_L -2d_M \leq g.
$$
It follows that 
$$
3d_L \geq d_L + 2d_M \geq d - g;
$$ 
hence we can replace $\frac{d-2g}{3}$ in \cite[formula (3.4)]{cl2} by $\frac{d-g}{3}$. It therefore remains to prove that
$$
\frac{2 \Cl_2(C)}{3} + \frac{d-g}{9} > \Cl_2(C)
$$
or equivalently
$$
\frac{d-g}{9} > \frac{1}{3} \Cl_2(C).
$$
This is equivalent to $d > g+\frac{3}{2} \left[ \frac{g-1}{2} \right] +6$.
\end{proof}

Combining everything, we get the following theorem.

\begin{theorem} \label{thm2.5}
If $g \geq 16$, there exists a curve $C$ satisfying \eqref{eq:new1} such that either
$$
\Cl_3(C) > \Cl_2(C)
$$
or there exists a semistable bundle $E$ of degree $d < 2g-2 + \frac{1}{2}\left[ \frac{g-1}{2} \right]$
which possesses no proper subbundle $F$ with $h^0(F) \geq \rk F + 1$ such that
$$
\Cl_2(C)\ge\gamma(E) \geq \frac{d_9}{3} -2.
$$
If $g\ge19$, this holds for every curve $C$ satisfying \eqref{eq:new1}.
\end{theorem}

\begin{proof} 
The theorem follows from Propositions \ref{prop2.1}, \ref{prop2.2}, \ref{prop2.8} and \ref{prop2.4} and Remarks \ref{rem18} and \ref{rem2.3}. We need only to check that the lower bounds \eqref{eqn2.1} and \eqref{eqdg} for $d$ are less than the upper bound of Proposition \ref{prop2.1}.
\end{proof}

\begin{lem} \label{lem2.6}
If $14 \leq g \leq 24$, then 
\begin{equation}\label{eq:new3}
\frac{d_9}{3} -2> \Cl_2(C).
\end{equation}
\end{lem}

\begin{proof}
By \eqref{dr}, we have
$$
d_9 \geq \min \left\{\left[ \frac{g-1}{2} \right] + 18,g+8  \right\}.
$$
The assertion follows from a simple computation.
\end{proof}

\begin{theorem} \label{thm2.9}
If $16 \leq g \leq 24$, then there exists a curve $C$ of genus $g$ such that
$$
\Cl_3(C) > \Cl_2(C).
$$
\end{theorem}

\begin{proof}
This follows at once from Theorem \ref{thm2.5} and Lemma \ref{lem2.6}.
\end{proof}

\begin{rem}\label{rem2.10}
{\em It is possible that \eqref{eq:new3} holds for other values of $g$, indeed for all $g\ge14$. If this is so, one can extend Theorem \ref{thm2.9} accordingly.}
\end{rem}

\section{An improved lower bound}\label{improved}

In this section we shall improve the lower bound of \cite[Theorem 4.1]{cl2}. We have already remarked in the proof of \cite[Theorem 4.6(ii)]{lmn} that
$$
\Cl_3(C) \geq \min\left\{ \frac{d_9}{3} -2, \frac{2\Cl(C) + 1}{3}, \frac{2 \Cl_2(C) + 2}{3} \right\}.
$$
For $\Cl_2(C) < \Cl(C)$ this is an improvement. We consider here the case $\Cl_2(C) = \Cl(C)$. Note that this is true for $\Cl(C)\le4$ by Lemma \ref{l2.1}, for all smooth plane curves \cite[Proposition 8.1]{cliff} and for the general curve of genus $\leq 19$ (see \cite[Theorem 1.7]{fo} for the case $g\le16$).

\begin{theorem} \label{thm3.1}
Let $C$ be a curve of genus $g \geq 7$ such that $\Cl_2(C) = \Cl(C) \geq 2$. Then 
$$
\Cl_3(C)  \geq \min \left\{ \frac{d_9}{3} -2, \frac{2\Cl(C) +2}{3} \right\}.
$$Moreover, if $\Cl_3(C)<\Cl(C)$, then any bundle computing $\Cl_3(C)$ is stable.
\end{theorem}

 We may assume $\Cl(C) \geq 3$ by Lemma \ref{l2.1}. We use the proofs in Sections 2 and 3 of \cite{cl2} making necessary improvements and proceed by a sequence of lemmas and propositions.
We follow the argument of \cite{cl2}.
Suppose throughout that $E$ is a bundle computing $\Cl_3(C)$.

\begin{lem} \label{lem3.2}
If $E$ has a line subbundle $F$ with $h^0(F) \geq 2$ and $d \leq 2g+6$, then 
$$
\gamma(E) \geq \frac{2\Cl(C)+2}{3}.
$$
\end{lem}

\begin{proof}
By  \cite[Lemma 2.2]{cl2} we know that 
$$\gamma(E) \geq \min\left\{\frac{2\Cl(C) + 1}{3},\frac13(4\Cl(C)+2g+2-d)\right\}.$$ 
We need first to improve the estimate in case (i) in the proof of this lemma. We have by \eqref{dr}
$$d_6 \geq \min \{ \Cl(C) +12,g+5 \} > \Cl(C) + 7.
$$ So
\begin{eqnarray*}
\gamma(E) \geq \frac{\Cl(C)}{3} + \frac{d_6}{3} -2 & >  & \frac{2\Cl(C) + 1}{3}.
\end{eqnarray*}
It is therefore sufficient to show that 
$$
\frac{1}{3} ( 4 \Cl(C) + 2g +2 -d) \geq \frac{1}{3} ( 2 \Cl(C) + 2).
$$
This is true provided $\Cl(C) \geq 3$ and $d \leq 2g+6$.
\end{proof}

\begin{lem} \label{lem3.3}
If $E$ has a subbundle $F$ of rank $2$ with $h^0(F) \geq 3$ and no line subbundle with $h^0 \geq 2$, and $d \leq 2g+2$, then
$$
\gamma(E) \geq \frac{2\Cl(C)+2}{3} .
$$
\end{lem}

\begin{proof}
We use \cite[Lemma 2.3]{cl2}. We need only to note that the estimate \cite[formula (2.3)]{cl2} can be improved to give the required result. For this improvement, observe that 
$$\gamma(E)\ge\frac{2t+g-1}3\ge\frac{2\Cl(C)+2}3$$
since $t=h^0(F)-2\ge1$.

\end{proof}

\begin{lem} \label{lem3.4}
Suppose that $E$ has a proper subbundle of maximal slope and rank $1$, and $d \geq 2g+4$. Then
$$
\gamma(E) \geq \frac{2\Cl(C)+2}{3}.
$$
\end{lem}

\begin{proof}
This is an immediate consequence of \cite[Lemma 3.1]{cl2}, since $3\Cl_3(C)$ is an integer.
\end{proof}

\begin{lem} \label{lem3.5}
Suppose that $g = 8$ or $g \geq 10$ and that $E$ has a proper subbundle of maximal slope and rank $2$, and $d \geq 2g+3$. Then
\begin{equation*} 
\gamma(E) \geq \frac{2\Cl(C)+2}{3}.
\end{equation*}
\end{lem}

\begin{proof}
We use \cite[Lemma 3.2]{cl2}. We need to check that 
\begin{equation} \label{eq3.1}
\frac{d + 2g -12}{9} > \frac{2\Cl(C)+1}{3}.
\end{equation}
This holds for $d \geq 2g +3$ if $g =8$ or $g \geq 10$.
\end{proof}

\begin{prop} \label{prop3.6}
Let $C$ be a curve of genus $g =8$ or $g \geq 10$ such that $\Cl_2(C) = \Cl(C) \geq 3$ and let $E$ be a bundle computing $\Cl_3(C)$. Then 
$$
\gamma(E)  \geq \min \left\{ \frac{d_9}{3} -2, \frac{2\Cl(C) +2}{3} \right\}.
$$
\end{prop}

\noindent
\begin{proof}
If $E$ does not possess a proper subbundle $F$ with $h^0(F) \geq \rk F +1$, then 
$$
\gamma(E) \geq \frac{d_9}{3} -2
$$ 
by Lemma \ref{pr}. So suppose $E$ does have such a subbundle and
$$
\gamma(E) \leq \frac{2\Cl(C) + 1}{3}.
$$ 
This gives a contradiction by Lemmas \ref{lem3.2} and \ref{lem3.3} if $d \leq 2g+2$ and by Lemmas \ref{lem3.4} and \ref{lem3.5} if $d \geq 2g +4$.

If $d = 2g+3$, then Lemma \ref{lem3.2} implies that $E$ has no line subbundle with $h^0 \geq 2$. Let $F$ be a subbundle of rank 2 with $h^0(F) \geq 3$. Then $F$ possesses a line subbundle $L$ with $h^0(L) = 1$, so $h^0(F/L) \geq 2$ and hence $d_{F/L} \geq d_1 > 2$. This implies $\mu(F) > 1$.

By Lemma \ref{lem3.5} all proper subbundles of $E$ of maximal slope are line bundles. Choose such a line bundle $L$ and consider the proof of Lemma \ref{lem3.4}, i.e. of \cite[Lemma 3.1]{cl2}. In order to get $\gamma(E) = \frac{2\Cl(C) + 1}{3}$, we must have equality in \cite[formula (3.4)]{cl2}, i.e. $d_L = 1$. So $L$ is not of maximal slope, a contradiction.
\end{proof}

The cases $g = 7$ and $g = 9$ require further arguments, because \eqref{eq3.1} can fail.

\begin{prop} \label{prop3.7}
Let $C$ be a curve of genus $g = 7$ with $\Cl(C) = 3$ and $E$ a bundle computing $\Cl_3(C)$. Then 
$$ 
\gamma(E) \geq \frac{8}{3}.
$$
\end{prop}

\begin{proof}
Recall that $\Cl_2(C) =3$ by \cite[Proposition 3.8]{cliff}. Moreover $d_9 = 16$. So $\frac{d_9}{3}-2 > 3$.

Note that $d \leq 3g-3 =18$. The proof of the theorem works for $d \leq 2g+2 =16$. So we are left with the cases $d = 17$ and 18.

If $d = 18$, \cite[formula (2.4)]{cl2} gives $\gamma(E) \geq \frac{2\Cl_2(C)}{3} = 2$. Moreover, $\gamma(E) \geq \frac{8}{3}$ unless $h^0(E) = 9$, in which case $\gamma(E) =2$. This contradicts \cite[Proposition 3.3]{cl2}.

If $d = 17$, we can assume that $E$ has no line subbundle with $h^0 \geq 2$ by Lemma \ref{lem3.2}. The only case in which we can have $\gamma(E) = \frac{2\Cl(C) + 1}{3}$ is when \cite[formula (2.4)]{cl2} is an equality. This implies that
$E$ fits into an exact sequence 
$$
0 \ra F \ra E \ra E/F \ra 0
$$
with $F$ of rank $2$ and degree $d_2=7$ with $h^0(F) = 3, \; d_{E/F} = 10$, $h^0(E/F) = 5$ and all sections of $E/F$ lift to $E$.

Since $h^0(E) =8$, there exists a line subbundle $L \subset E, \; d_L \geq 2$ and $h^0(L) = 1$. This cannot be a subbundle of maximal slope, since this would require equality in \cite[formula (3.4)]{cl2}, which means $d_L = 1$. So there exists a subbundle $G$ of maximal slope with rank 2. If $d_G \geq 8$, then by the proof of \cite[Lemma 3.2]{cl2}, $\gamma(G) \geq 3$ and also $\gamma(E/G) \geq 3$. So $\gamma(E) \geq 3$, a contradiction.  Hence $F$ is a subbundle of maximal slope. 

Now $d_{E/L} = 17 -d_L$. If $M$ is a subbundle of $E/L$, the pullback to $E$ has degree $d_M + d_L \leq 7$. So 
$$
d_M \leq 7 - d_L < \frac{17 -d_L}{2}
$$
and $E/L$ is stable.

Note that $h^0(E/L) \geq 7$, so $h^1(E/L) \geq 7 + 12 - d_{E/L} \geq 4$. Hence either $E/L$ or $K \otimes (E/L)^*$ contributes to $\Cl_2(C)$. Since $\Cl_2(C) = 3$, this gives 
$d_{E/L} - 2(h^0(E/L) -2) \geq 6$, i.e.
$$
h^0(E/L) \leq  \frac{d_{E/L}}{2} -1 \leq \frac{13}{2},
$$
a contradiction.
\end{proof}

\begin{prop} \label{prop3.8}
Let $C$ be a curve of genus $g = 9$ with $\Cl(C) \geq 3$ and $E$ a bundle computing $\Cl_3(C)$. Then 
\begin{itemize}
\item either $\Cl(C) =3$ and $\gamma(E) \geq \frac{8}{3}$
\item or $\Cl(C) =4$ and $\gamma(E) = \frac{10}{3}$.
\end{itemize}
\end{prop}

\begin{proof}
Recall that $d_9 = 18$. So $\frac{d_9}{3} -2 \geq \Cl(C)$.

The only case we need to consider is $d = 2g+3 =21$. If $\Cl(C) = 3$, then \eqref{eq3.1} holds and the proof goes through as for Proposition \ref{prop3.6}.

So suppose $\Cl(C) = 4$. By Lemma \ref{l2.1}, $\Cl_2(C) = 4$. In this case $\gamma(E)\le\frac{10}3$ by \cite[Theorem 4.3]{lmn}, so we can assume that $E$ has no line subbundle with $h^0\ge2$ by Lemma \ref{lem3.2}. However, we can have equality in \cite[formula (2.4)]{cl2}. Thus $E$ fits into an exact sequence
$$
0 \ra F \ra E \ra E/F \ra 0
$$
with $F$ of rank $2$ and degree $d_2=8$ with $h^0(F) = 3, \; d_{E/F} = 13$, $h^0(E/F) = 6$ and all sections of $E/F$ lift to $E$.

We argue similarly as in the proof of Proposition \ref{prop3.7}. Since $h^0(E) = 9$, there exists a line subbundle $L$ with $d_L \geq 3$ and $h^0(L) = 1$. 
Again no line subbundle of $E$ can be of maximal slope and if $G$ is a subbundle of rank 2 of maximal slope with $d_G \geq 9$, then the proof of \cite[Lemma 3.2]{cl2} gives $\gamma(G) \geq \frac{7}{2}$ and $\gamma_{E/G} \geq 4$. So 
$$
\gamma(E) \geq \frac{11}{3},
$$ 
contradicting the fact that $\gamma(E)\le\frac{10}3$. 
Hence $F$ is a subbundle of maximal slope.

Note that $h^0(E/L) \geq 8$. Arguing similarly as above we get
$$
h^0(E/L) \leq \frac{d_{E/L}}{2} -2 \leq 7,
$$
a contradiction.
So we do not have equality in \cite[formula (2.4)]{cl2}, which implies $\gamma(E) \geq \frac{10}{3}$. Since $\Cl_3(C) \leq \frac{10}{3}$, this gives the result.
\end{proof}

As an immediate consequence we get

\begin{cor}\label{corg=9}
Let $C$ be a curve of genus $g =9$ with $\Cl(C) = 4$. Then
$$
\Cl_3(C) = \frac{10}{3}.
$$
\end{cor}

\begin{proof}[Proof of Theorem \ref{thm3.1}] The inequality for $\Cl_3(C)$ is a consequence of Propositions \ref{prop3.6},
 \ref{prop3.7} and  \ref{prop3.8}. The last assertion follows from Proposition \ref{newprop}.
\end{proof}

\begin{rem} \label{rem3.10}
{\em
Suppose that $C$ is as in the statement of Theorem \ref{thm3.1} and further that $3\Cl(C)  \geq 2d_2-6$. Let $L_1$ and $L_2$ be line bundles on $C$ of degree $d_2$ with $h^0(L_i) = 3$ for $i = 1,2$. By Lemma \ref{l2.4}, $E_{L_1}$ and $E_{L_2}$ are stable with $h^0 = 3$.
 By Proposition \ref{p2.5}, all non-trivial extensions
$$
0 \ra E_{L_1} \ra E \ra L_2 \ra 0
$$
with $h^0(E) = 6$ give semistable bundles $E$ and such bundles exist if and only if 
$$
h^0(E_{L_1} \otimes E_{L_2}) \geq 10.
$$
Moreover,
$$
\gamma(E) = \frac{2d_2 -6}{3}.
$$
For the general curve of genus $9$ or $11$, these values are attained \cite{lmn}.

Note that, if in addition $d_2$ computes $\Cl(C)$, then
$$
\gamma(E) = \frac{2\Cl(C) +2}{3}.
$$
Smooth plane curves satisfy these conditions and we have a more precise statement in the next section (Theorem \ref{thm5.6}). The normalisations of nodal plane curves with small numbers of nodes are also covered by this remark (see Theorem \ref{thmnodal} and Remark \ref{rem5.7}).}
\end{rem}

\begin{rem}
{\em
For a general curve $C$ of genus $g$ we have
$$
d_9 = g + 9 - \left[ \frac{g}{10} \right].
$$
So $\frac{d_9}{3} - 2 \geq \frac{2\Cl(C) + 2}{3}$ for $g \leq 30$. If $\Cl_2(C) = \Cl(C)$ for such curves, then
$$
\Cl_3(C) \geq \frac{2\Cl(C) +2}{3}.
$$
For instance, for a general curve of genus $10$, we have $\Cl_2(C)=\Cl(C)=4$, so $\frac{10}3\le\Cl_3(C)\le4$. For a general curve of genus $11$, we know that $\Cl_2(C)=\Cl(C)=5$ \cite[Theorem 1.3]{fo2}; so, using \cite[Theorem 4.6]{lmn}, we obtain  $4\le\Cl_3(C)\le\frac{14}3$, which is an improvement on the known result $\frac{11}3\le\Cl(C)\le\frac{14}3$}.
\end{rem}

\section{Plane curves}\label{plane}

To begin with, let $C$ be a smooth plane curve of degree $\delta \geq 6$ and let $H$ denote the hyperplane bundle on $C$. We know that 
$$
\Cl_2(C) = \Cl(C) = \delta -4
$$
(see \cite[Proposition 8.1]{cliff}). We also know the values of all $d_r$ by Noether's Theorem (a proof, which also works for any integral plane curve as claimed by Noether, was given by Hartshorne \cite[Theorem 2.1]{h}). In particular,
$$
d_1 = \delta - 1, \quad d_2 = \delta, \quad d_6=3\delta-3, \quad d_9 = 3\delta.
$$
Moreover, by the same theorem, $H$ is the only line bundle of degree $\delta$ on $C$ with $h^0(H)=3$ and also the only line bundle computing $\Cl(C)$.
 
The following proposition is a consequence of Theorem \ref{thm3.1}.

\begin{prop} \label{prop4.1}
Let $C$ be a smooth plane curve of degree $\delta \geq 6$. Then
$$
\Cl_3(C) \geq \frac{2\delta -6}{3}.
$$
\end{prop}

\begin{proof}
The result follows from Theorem \ref{thm3.1}, since
\begin{equation} \label{e5.1}
\frac{d_9}{3} -2 = \delta -2 > \frac{2\delta -6}{3}.
\end{equation}
\end{proof}

Note that, if $\delta=6$, we have equality in Proposition \ref{prop4.1} since $\Cl_3(C)=\Cl(C)=2$ by Lemma \ref{l2.1}. So we can assume $\delta\ge7$.

Suppose now that $E$ is a bundle computing $\Cl_3(C)$.

\begin{lem} \label{l5.2}
If $d \geq 2 \delta + 6$ and $E$ has a subbundle $F$ of maximal slope with $\rk F=1$, then
$$
\gamma(E) > \frac{2 \delta -6}{3}.
$$
\end{lem}

\begin{proof}
We follow the proof of \cite[Lemma 3.1]{cl2}. First observe that
$$
\frac{\Cl(C)}{3} + \frac{2d  -6}{9} > \frac{2 \delta-6}{3}.
$$
So we can use \cite[formulas (3.4) and (3.5)]{cl2} obtaining
$$
\gamma(E) \geq \frac{2 \Cl_2(C)+d_F}{3}.
$$
To get $\gamma(E) = \frac{2\delta -6}{3}$, this requires $d_F \leq 2$.

On the other hand, if $d \geq 2 \delta +6$ and $\gamma(E)=\frac{2\delta-6}3$, then 
$$
h^0(E) =\frac{d-3\gamma(E)}2+3= \frac{d}{2} - \delta + 6 \geq 9.
$$
So $E$ possesses a line subbundle of degree $\geq 3$, a contradiction.
\end{proof}

\begin{lem} \label{l5.3}
If $d > g + \frac{3}{2} \delta$, then 
$$
\gamma(E) > \frac{2 \delta -6}{3}.
$$
\end{lem}

\begin{proof}
Note first that $d>g+\frac32\delta$ implies that $d\ge2\delta+6$. By Lemma \ref{l5.2}, we can therefore assume that every subbundle $F$ of $E$ of maximal slope has rank $2$. We check now that all the numbers in the minimum of \cite[Lemma 3.2]{cl2} are $>\frac{2\delta-6}3$. For the first number, this is immediate since $\Cl_2(C)=\Cl(C)=\delta-4$. The second requires precisely the condition $d>g+\frac32\delta$. The third needs only $d>6$. For the fourth, we need $d<4g-12$, which is true since $d\le3g-3$ and $g>9$. Finally, for the fifth number, we need $d>6\delta-2g-6$, which is easily seen to be true. 
\end{proof}

\begin{lem} \label{l5.4}
If $E$ has a line subbundle with $h^0 \geq 2$ and $d < 2 \delta -8 + 2g$, then 
$$
\gamma(E) > \frac{2\delta -6}{3}.
$$
\end{lem}

\begin{proof}
Since $d_6 = 3 \delta -3>\delta+4$, we see from the proof of \cite[Lemma 2.2]{cl2} that $\gamma(E) > \frac{2\delta -6}{3}$.
\end{proof}

\begin{lem} \label{l5.5}
Suppose that $E$ is a bundle computing $\Cl_3(C) = \frac{2 \delta - 6}{3}$.
If $d \leq 2g+1$, then $E$ fits into a non-trivial exact sequence
$$
0 \ra E_H \ra E \ra L \ra 0
$$
where $L \simeq H$ or $\simeq H^{\delta -4}$  and all sections of $L$ lift to $E$.
\end{lem}

\begin{proof}  Since $2g+1<2\delta-8+2g$, it follows from \eqref{e5.1} and Lemma \ref{l5.4} that $E$ has a subbundle $F$ of rank $2$ with $h^0(F) \geq 3$ and no line subbundle with $h^0 \geq 2$.

We follow the proof of \cite[Lemma 2.3]{cl2}. In the case $d_{2t} < 2t +g -1$ and $d_u < u + g -1$, the only possibility is that all the inequalities are equalities. This gives $t = 1$ (hence $h^0(F)=3$), $d_F = d_2 = \delta$ and $d_u=\delta-4+2u$. For any line subbundle $M$ of $F$ we have $h^0(M) \leq 1$.
So $h^0(F/M) \geq 2$. Hence $d_{F/M} \geq d_1 = \delta -1$. So $d_M \leq 1$ and $F$ is stable.

As in the proof of Proposition  \ref{p2.5} we see that $F$ is generated and has the form $F \simeq E_N$ for some line bundle $N$ of degree $d_2$ with $h^0(N) = 3$. The only such bundle is $H$. Moreover, $L:=E/E_H$ is a line bundle such that either $L$ or $K_C\otimes L^*$ computes $\Cl(C)$. It follows from Noether's Theorem that either $L\simeq H$ or $L\simeq K_C \otimes H^* \simeq H^{\delta -4}$.

In the argument leading up to \cite[formula (2.3)]{cl2}, we have the inequality $\frac{\gamma(E)}{2}\ge\frac{t}3+\frac{g-1}6$. This gives $\gamma(E)\ge\frac{g+1}{3}$ which implies that $\gamma(E) > \frac{2 \delta -6}{3}$.

Finally, for \cite[formula (2.4)]{cl2}, we obtain $\gamma(E) > \frac{2 \delta -6}{3}$ provided $d \leq 2g+1$.
\end{proof}

\begin{theorem} \label{thm5.6}
If $C$ is a smooth plane curve of degree $\delta \geq 7$ and $\Cl_3(C) = \frac{2 \delta -6}{3}$,
then any bundle $E$ computing $\Cl_3(C)$ is stable and fits into an exact sequence 
\begin{equation} \label{eqn3.1}
0 \ra E_H \ra E \ra H \ra 0
\end{equation}
and all sections of $H$ lift to $E$. Moreover, such extensions exist if and only if $h^0(E_H \otimes E_H) \geq 10$.
\end{theorem}

\begin{proof}
Stability of $E$ follows from Proposition \ref{newprop}. Next we eliminate the possibility $L \simeq H^{\delta - 4}$ in Lemma \ref{l5.5}.
In this case $d = 2g-2$ and we can check that 
$$2g-2 > g + \frac{3}{2} \delta
$$
for $\delta \geq 7$. It follows from Lemma \ref{l5.3} that $\gamma(E) > \frac{2\delta -6}{3}$, a contradiction.

Since $2g+1>g+\frac32\delta$, Lemmas \ref{l5.3} and \ref{l5.5} cover all possibilities for $d$. This implies the existence of \eqref{eqn3.1}. The last assertion follows from Proposition \ref{p2.5}.
\end{proof}

We now consider the case when $C$ is the normalisation of a plane curve $\Gamma$ of degree $\delta$ whose only singularities are $\nu$ simple nodes. Since Noether's Theorem applies to $\Gamma$ rather than $C$, we cannot use it directly to obtain information about $C$. However many relevant facts are known about $C$. 

For our purposes, we shall assume that the nodes are in general position and that
\begin{equation}\label{eqnu} 
1\le\nu \leq \frac{1}{2} ( \delta^2 -7\delta +14).
\end{equation}
Note that, since $C$ has genus $g=\frac12(\delta-1)(\delta-2)-\nu$, \eqref{eqnu} is equivalent to
\begin{equation}\label{eqgenus}
g\ge2\delta-6.
\end{equation}
By \cite{c} and \cite[Corollary 2.3.1]{cm}, we have $\Cl(C) = \delta -4$ and this is computed by both $d_1$ and $d_2$. Moreover there are finitely many line bundles $H_1,\ldots,H_\ell$ of degree $d_2=\delta$ with $h^0(H_i)=3$; in fact, this is true for $g\ge\frac32\delta-3$ (or equivalently $\nu\le\frac12(\delta^2-6\delta+8)$) by \cite[Section 4]{s}. (For $g<\frac32\delta-3$, the result must fail since this is equivalent to the Brill-Noether number for line bundles of degree $\delta$ with 3 independent sections on $C$ being positive.)

We shall make the additional assumption that
\begin{equation}\label{eqd4}
d_4 \geq 2\delta -4;
\end{equation}
it follows then by \cite[Theorem 5.2]{cliff} that
\begin{equation}\label{eqcl2}
\Cl_2(C)=\Cl(C)=\delta-4.
\end{equation}
\begin{rem} \label{rem5.7}
{\em For $\delta\ge7$, we certainly have $d_4 \geq \delta +4$ by \eqref{dr}. So \eqref{eqd4} is satisfied for $\delta = 7$ or 8. The formula \eqref{eqd4} also holds for $\nu\le4$. To see this it is sufficient to show that any line bundle $L$ of degree $2\delta-5$ has $h^0(L)\le4$. For this we can write $\pi: C\to\Gamma$ for the normalisation map and apply \cite[Theorem 2.1]{h} to the torsion-free sheaf $\pi_*(L)$ which has degree $2\delta-5+\nu\le2\delta-1$. When $\nu\le3$, we obtain immediately $h^0(L)=h^0(\pi_*(L))\le4$. If $\nu=4$, we note that $\pi_*(L)$ is not of the required form for $h^0(\pi_*(L))=5$.}
\end{rem}

Before proceeding to our main result, we shall prove a lemma which we shall also need in Section \ref{three}.
\begin{lem}\label{lemg9}Let $C$ be a curve of genus $9$ with $\Cl(C)=3$. Suppose that $E$ is a semistable bundle of rank $3$ and degree $24$ with $h^0(E)\ge4$. Then $\gamma(E)\ge3$. 
\end{lem} 
\begin{proof} Since $\Cl(C)\le4$, we have $\Cl_2(C)=\Cl(C)$ by Lemma \ref{l2.1}; moreover $d_9\ge17$ by \eqref{dr}. So, by Theorem \ref{thm3.1}, $\gamma(E)\ge\frac83$. If $\gamma(E)=\frac83$, then clearly $h^0(E)=\frac{d-3\gamma(E)}2+3=11$, so $E$ possesses a line subbundle of degree at least $3$. If this is a subbundle of maximal slope, then, by \cite[Lemma 3.1]{cl2} and its proof (see in particular \cite[ formula (3.4)]{cl2}), $\gamma(E)\ge3$, a contradiction. So every subbundle $F$ of $E$ of maximal slope must have $\rk F=2$. 

We now consider the proof of \cite[Lemma 3.2]{cl2}. The first three numbers and the last number in the minimum are certainly $\ge3$. The fourth number, however, is $\frac{8}{3}$. We can have $\gamma(E) = \frac{8}{3}$ if and only if all inequalities leading up to this are equalities. This implies that 
$$
F \; \mbox{computes} \; \Cl_2(C), \quad h^1(E/F) = 1, \quad d_{E/F} = 14.
$$ 
So $d_F = 10$. Since $E$ has no line subbundle of maximal slope, the maximal slope of a line subbundle of $F$ is 4. So $F$ has no line subbundle with $h^0 \geq 2$. By Lemma \ref{pr} this implies that $d_F \geq d_4$ and so $\geq 11$ by \eqref{dr}. This is a contradiction.
\end{proof}

\begin{theorem}\label{thmnodal}
Suppose that $C$ is the normalisation of a nodal plane curve of degree $\delta\ge7$ with $\nu$ nodes in general position and that \eqref{eqnu} holds. Suppose further that \eqref{eqd4} holds and $g\ge9$. Then 
$$\Cl_3(C)\ge\frac{2\delta-6}3.$$ 
Moreover, if $\Cl_3(C)=\frac{2\delta-6}3$, then any bundle $E$ computing $\Cl_3(C)$ is stable and fits into an exact sequence 
\begin{equation}\label{eqhi}
0\to E_{H_i}\to E\to L\to 0,
\end{equation}
where $3\le h^0(L)\le g+4-\delta$, $d_L=\delta-6+2h^0(L)$ and all sections of $L$ lift to $E$.
\end{theorem}
\begin{proof} Note first, using \eqref{eqd4}, that $d_9\ge d_4+5 \geq 2 \delta +1$; so \eqref{e5.1} holds. Since \eqref{eqcl2} also holds, Proposition \ref{prop4.1} is valid with the same proof as before; so $\Cl_3(C)\ge\frac{2\delta-6}3$. 

Suppose now that $\Cl_3(C)=\frac{2\delta-6}3$ and that $E$ is a bundle computing $\Cl_3(C)$. The proof of Lemma \ref{l5.2} remains valid. For Lemma \ref{l5.3}, we need first that $d>g+\frac32\delta$ implies that $d\ge2\delta+6$.This follows from \eqref{eqgenus} for $\delta\ge8$ and can easily be checked for $\delta=7$ and $g\ge9$. The condition $d<4g-12$ holds for $d\le3g-3$ provided $g>9$. For $g=9$ (which requires $\delta=7$ by \eqref{eqgenus}), the condition still holds for $d<3g-3$; the case $d=3g-3$ is covered by Lemma \ref{lemg9}.   Finally $d>6\delta-2g-6$ holds for $g>2\delta-6$ since $d\ge2\delta+6$; when $g=2\delta-6$, the condition $d>g+\frac32\delta$ implies $d>6\delta-2g-6$ for $\delta\ge8$.

For Lemma \ref{l5.4}, the requirement is $d_6>\delta+4$, which follows from \eqref{dr} and \eqref{eqgenus}.  It follows that every subbundle of maximal slope of $E$ has rank 2, so we can apply Lemma \ref{l5.5}. There is a minor change in the proof since $d_1=\delta-2$, which means that $F$ can have a line subbundle of degree $2$; however this does not affect the argument. On the other hand, the hyperplane bundle $H$ is no longer unique and we do not know all the bundles computing $\Cl(C)$, so we just obtain the form \eqref{eqhi} for the exact sequence defining $E$.

The remaining problem is that Lemmas \ref{l5.3} and \ref{l5.5} may no longer cover all cases. In fact $d\le g+\frac32\delta$ implies $d\le2g+2$ under our assumptions, but it is possible to have $d=2g+2$ for low values of $\delta$. In this case, we need to reexamine \cite[formula (2.4)]{cl2}; if $d=2g+2$, we still require $t=1$ and hence $d_F=d_2=\delta$; the quotient line bundle $L=E/F$ no longer computes the Clifford index, but it is still the case that $\gamma(L)=\delta-4$, giving a sequence of the form \eqref{eqhi}. The stability of $E$ follows from Proposition \ref{newprop}, while the inequalities for $h^0(L)$ come from $h^0(E)\ge6$ and $d\le2g+2$. \end{proof}

\begin{rem}\label{2g+2}
{\em The only case in which the possibility $d=2g+2$ needs to be included in \eqref{eqhi} under the assumptions of the theorem is when $\delta=8$, $g=10$. For small numbers of nodes, other possibilities can be excluded; for example, when $\delta=7$ and $\nu=1$ (so $g=14$), we have $2g-2>g+\frac32\delta$. We can therefore assume $d\le 2g-4$ in \eqref{eqhi}, corresponding to $h^0(L)\le8$.}\end{rem}

\begin{rem}\label{g=7}
{\em The excluded case $\delta=7$, $g=8$ will be covered in Section \ref{three} (Proposition \ref{prop5.6}), as will the case $\delta=7$, $g=7$ (hence $\nu=8$). In the latter case, it is proved in \cite{c} that $d_1=5$ but this does not imply that $\Cl(C)=3$ since there are infinitely many pencils on $C$ with degree $5$. Thus Theorem \ref{thmnodal} does not apply, but a modified version does hold (see Proposition \ref{prop5.5}), perhaps under stronger generality conditions.}\end{rem}

\section{Curves with Clifford index three}\label{three}

Let $C$ be a curve of genus $g$ with $\Cl(C) = 3$ and hence $g \geq 7$. 
We have $d_9 \geq 16$ for $g \geq 8$ from \eqref{dr}. For $g = 7,\; d_9 = 16$ by Riemann-Roch.
By Theorem \ref{thm3.1}, we have
$$
\frac{8}{3} \leq \Cl_3(C) \leq 3.
$$
Hence any bundle $E$ computing $\Cl_3(C)$ possesses a proper subbundle $F$ with $h^0(F) \geq \rk F + 1$. 

We now consider the possibility that $\gamma(E) = \frac{8}{3}$. Note that this can happen only if $d$ is even. 
We suppose throughout that $E$ is a bundle of degree $d$ computing $\Cl_3(C)$.

\begin{lem}  \label{lem5.1}
If $E$ possesses a line subbundle with $h^0 \geq 2$ and $d \leq 2g +5$, then
$$
\gamma(E) \geq 3.
$$
\end{lem}

\begin{proof} Consider the proof of \cite [Lemma 2.2]{cl2}. Noting that $d_6 \geq 12$ by \eqref{dr}, we see that the only possibility for having $\gamma(E) <3$ in the proof of \cite[Lemma 2.2]{cl2} is the inequality 
$$
\gamma(E) \geq \frac{1}{3}(4\Cl(C) + 2g + 2 -d) = 4 + \frac{1}{3}(2g + 2 - d).
$$
This gives $\gamma(E) \geq 3$.
\end{proof} 

\begin{rem}  \label{rem5.3}
{\em
Since we always have $d \leq 3g-3$, the assumption $d \leq 2g+5$ is redundant for $g = 7$ and $g=8$.
}
\end{rem}

\begin{lem} \label{lem5.3}
Suppose that there exists an exact sequence 
\begin{equation} \label{eq5.1}
0 \ra F \ra E \ra E/F \ra 0
\end{equation} 
with $\rk F =2$ and $h^0(F) \geq 3$ and that $E$ has no line subbundle with $h^0 \geq 2$.
If $d \leq 2g +2$ and $\gamma(E) = \frac{8}{3}$, then
$$
d_F = d_2 = 7, \; h^0(F) = 3, \; h^0(E/F) \geq 3 \; \mbox{and} \; d_{E/F} = 1 + 2h^0(E/F).
$$
Moreover, all sections of $E/F$ lift to $E$. 
\end{lem}

\begin{proof}
We follow the proof of \cite[Lemma 2.3]{cl2}. 
The first case to be considered is when $d_{2t} < 2t +g -1$ and $d_u < u + g -1$. Then we have $\gamma(E) = \frac{8}{3}$ only if $t=1$ (hence $h^0(F)=3$), $d_2 = d_F$ and $d_u = d_{E/F}$; moreover $d_2=\Cl(C)+4 = 7$ and $d_u=\Cl(C)+2u  = 1 + 2h^0(E/F)$. Since $h^0(E)\ge6$, we have also $h^0(E/F)\ge3$.
Moreover, $d = 10 + 2u$; since $\gamma(E) = \frac{8}{3}$, this gives $h^0(E) = 4 + u = h^0(F) + h^0(E/F)$. Hence all sections of $E/F$ lift to $E$.

The case of \cite[formula (2.3)]{cl2} can give $\gamma(E) = \frac{8}{3}$ only if $t = 1$.
In this case the hypothesis $d_{2t} \geq 2t +g-1$ gives $d_2 \geq g+1$ which is impossible.
This leaves us with the case of \cite[formula (2.4)]{cl2}. If $d \leq 2g+1$, this gives $\gamma(E) \geq 3$. For $d = 2g+2$ we must have $t = 1, \; d_F = d_2 =7, \; u =g-4$ and $d_{E/F} = d_u = 2g -5$.
The result follows as in the first part of the proof.
\end{proof}

\begin{lem}  \label{lem5.4}
Suppose that there exists an exact sequence \eqref{eq5.1} with $\rk F =2$ and $h^0(F) \geq 3$ and that $E$ has no line subbundle with $h^0 \geq 2$. If $d = 2g+4$ and $\gamma(E) = \frac{8}{3}$, then $h^0(E) = g+1, \; h^0(F) = 3$, and 

either  $\bullet \; d_F = d_2 = 7,\; d_{E/F} = 2g -3, \; h^0(E/F) = g-2 \; \mbox{or} \; g-1$

or  \hspace{0.5cm} $\bullet \;d_F =8,\;d_{E/F} = 2g -4, \; h^0(E/F) = g-2$.
\end{lem}

\begin{proof}
\cite[Formula (2.4)]{cl2} gives 
$$
\gamma(E) \geq \frac{2 \Cl(C) +6t -6}{3}.
$$
For $\gamma(E) = \frac{8}{3}$ we still need $t=1$, so $h^0(F)=3$, but it is now possible that $d_F = 8$. 

We have $h^0(E) = g+1$, since $\gamma(E) = \frac{8}{3}$. Hence $h^0(E/F)  \geq g-2$. The rest follows from Riemann-Roch.
\end{proof} 

\begin{prop} \label{prop5.5}
Let $C$ be a  curve of genus $g=7$ with $\Cl(C) =3$ and suppose that $\Cl_3(C) = \frac{8}{3}$. Then $E$ is stable and fits into an exact sequence \eqref{eq5.1} with $h^0(F) = 3$. Moreover one of the following holds
 \begin{itemize}
\item $d_F =7,\; d_{E/F} = 7,\; h^0(E/F) =3, \; h^0(E) = 6$,
\item $d_F =7,\; d_{E/F} = 9,\; h^0(E/F) =4, \; h^0(E) = 7$,
\item $d_F =7,\; d_{E/F} = 11,\; h^0(E/F) =5$ or $6, \; h^0(E) = 8$,
\item $d_ F =8,\; d_{E/F} = 10,\; h^0(E/F) =5, \; h^0(E) = 8$.
\end{itemize}  
\end{prop} 

\begin{proof}
Stability of $E$ follows from Proposition \ref{newprop}. Since $d \leq 3g-3$, the rest follows from Lemmas \ref{lem5.1}, \ref{lem5.3} and \ref{lem5.4}.
\end{proof}

\begin{prop} \label{prop5.6}
Let $C$ be a curve of genus $g = 8$ with $\Cl(C) =3$ and suppose that $\Cl_3(C) = \frac{8}{3}$. Then $E$ is stable and fits into an exact sequence \eqref{eq5.1} with $h^0(F) = 3$. Moreover one of the following holds
 \begin{itemize}
\item $d_F =7,\; d_{E/F} = 7,\; h^0(E/F) =3, \; h^0(E) = 6$,
\item $d_F =7,\; d_{E/F} = 9,\; h^0(E/F) =4, \; h^0(E) = 7$,
\item $d_F =7,\; d_{E/F} = 11,\; h^0(E/F) =5, \; h^0(E) = 8$,
\item $d_F =7,\; d_{E/F} = 13,\; h^0(E/F) =6$ or $7, \; h^0(E) = 9$,
\item $d_ F =8,\; d_{E/F} = 12,\; h^0(E/F) =6, \; h^0(E) = 9$.
\end{itemize}                         
 For the general curve of genus $8$ only the last possibility can occur. 
\end{prop}

\begin{proof}
Stability of $E$ follows from Proposition \ref{newprop}. Since $d \leq 3g-3$, the various possibilities for \eqref{eq5.1} follow from Lemmas \ref{lem5.1}, \ref{lem5.3} and \ref{lem5.4}. For the last assertion note that the general curve of genus 8 has $d_2 =8$. 
\end{proof}

For $g \geq 9$ we need to consider the possibility that $d \geq 2g+6$. For this we use the results of \cite[Section 3]{cl2}.

\begin{prop}  \label{prop5.7}
Let $C$ be a curve of genus $g \geq 9$ with $\Cl(C) = 3$ and suppose that $\Cl_3(C) = \frac{8}{3}$.
Then $d_2 = 7,\; 14 \leq d \leq 2g$ and $E$ is stable and fits into an exact sequence \eqref{eq5.1} with
$$
\rk F =2,\; d_F = 7,\; h^0(F) =3, \;  d_{E/F} = d-7, \; h^0(E/F) = \frac{d-8}{2}
$$
and all sections of $E/F$ lift to $E$.
\end{prop}

\begin{proof}
Once again stability follows from Proposition \ref{newprop}. 

Suppose that $E$ possesses a subbundle $L$ of maximal slope of rank 1. The first and third numbers in the minimum of  \cite[Lemma 3.1]{cl2} are clearly $>\frac83$ (this requires only $d\ge11$). By \cite[formula (3.4)]{cl2}, we see that the second number can be replaced by $\frac{2\Cl_2(C)+d_L}3$, so we must have $d_L\le2$. It follows that $E$ has no line subbundle with $h^0 \geq 2$. Hence it is among the cases listed in Lemma \ref{lem5.4}. In particular every subbundle $F$ of maximal slope has rank $2$.

Let $F$ be such a subbundle and suppose $d \geq 2g+2$. The first 3 numbers in the minimum of the statement of \cite[Lemma 3.2]{cl2} are  $> \frac{8}{3}$ (the requirement for this is $d\ge g+11$).
The fourth number is $> \frac{8}{3}$ if and only if $d < 4g -12$. Since $d \leq 3g-3$, this holds always if $g \geq 10.$ For $g =9$ the fourth number is $> \frac{8}{3}$ for $d <3g-3$. The remaining case $g=9$, $d=24$ is covered by Lemma \ref{lemg9}. The last number is  $> \frac{8}{3}$ if and only if $d > 36 -2g$. This holds for $d \geq 2g+2$ if $g \geq 9$.

We are left with the case $d \leq 2g$. The result now follows from Lemma \ref{lem5.3}.
\end{proof}

\begin{theorem}\label{thm6.8}
Let $C$ be a curve of genus $g \geq 9$ with $\Cl(C) =3$. If $d_2 > 7$, and in particular if $g \geq 16$, then
$$
\Cl_3(C) = 3.
$$
For all $g \geq 9$ there exist curves with these properties.
\end{theorem}

\begin{proof}
The first assertion follows from Proposition \ref{prop5.7} once we know that $d_2 \geq 8$ whenever $g \geq 16$. In fact, if $d_2 =7$, then $C$ possesses as a plane model a septic. Hence $g \leq 15$. 

For $9 \leq g \leq 15$ note that by the Hurwitz formula the family of curves with Clifford index 3 is of dimension $2g + 5$. On the other hand, the family of plane septics of genus $g$ is of dimension $12 + g < 2g+5$. This proves the final statement.
\end{proof}

\begin{cor}
Let $C$ be a smooth complete intersection of $2$ cubics in $\PP^3$. Then
$$
\Cl_3(C) = \Cl(C) = 3.
$$
\end{cor}

\begin{proof}
It is known that $C$ has Clifford dimension 3, genus 10 and $\Cl(C) = 3$ (see \cite{elms}). In particular $d_2$ does not compute $\Cl(C)$. So $d_2 > 7$.
\end{proof}

The curves of this corollary are the only curves of Clifford dimension $\geq 3$ with $\Cl(C) = 3$ (see \cite{elms}).

\begin{rem} 
{\em
Suppose that $C$ is a curve of genus $g \geq 9$ with $\Cl(C) = 3$ and $d_2 = 7$. Then  $C$ possesses as a plane model a septic. For $g = 15$ this model is smooth and Theorem \ref{thm5.6} applies. In particular $\Cl_3(C) = \frac{8}{3}$ if and only if $h^0(E_H \otimes E_H) \geq 10$.

If $9 \leq g \leq 14$, then the general curve of this type is the normalisation of a nodal septic with nodes in general position, so Theorem \ref{thmnodal} applies and gives a somewhat more precise result.  
}
\end{rem}

\section{Coherent systems}\label{cs}

Recall that a {\it coherent system of type} $(n,d,k)$ on a curve $C$ is a pair $(E,V)$ where $E$ is a vector bundle of rank $n$ and degree $d$ on $C$ and $V$ is a linear subspace of $H^0(E)$ of dimension $k$. For any $\alpha > 0$ we define the $\alpha${\it -slope} of $(E,V)$  by
$$
\mu_{\alpha}(E,V) := \frac{d}{n} + \alpha \frac{k}{n}.
$$
The coherent system $(E,V)$ is called $\alpha${\it -stable} ($\alpha${\it -semistable}) if, for all proper coherent subsystems $(F,W)$ of $(E,V)$,
$$
\mu_{\alpha}(F,W) < (\leq)\  \mu_{\alpha}(E,V).
$$

\begin{prop} \label{prop7.1}
Suppose $E$ computes $\Cl_n(C)$ and $\Cl_r(C) \geq \Cl_n(C)$ for all $r\le n$. Then $(E,H^0(E))$ is $\alpha$-semistable for all $\alpha > 0$. If also $E$ is stable, then $(E,H^0(E))$ is $\alpha$-stable for all $\alpha > 0$.
\end{prop}

\begin{proof}
Write $h^0(E) = n + s$ with $s \geq n$. If $F$ is any subbundle of $E$, then $\mu(G) \leq \frac{d}{n}$ for any subbundle $G$ of $F$. We need to show that 
$$
\frac{h^0(F)}{\rk F} \leq \frac{n + s}{n}.
$$
If this is not true, then by \cite[Lemma 2.1]{cl4} we have
$$
\gamma(E) = \frac{d-2s}{n} > \min \left\{ \gamma(G) \; \left| \;   
\begin{array}{c}
G \; \mbox{semistable}, \rk G \leq n,\\
\frac{d_G}{\rk G} \leq \frac{d}{n}, \; \frac{h^0(G)}{\rk G} \geq \frac{n+s}{n}
\end{array} 
     \right\}.  \right.        
$$
All such $G$ contribute to $\Cl_{\rk G}(C)$. Since $\Cl_r(C) \geq \Cl_n(C)$ for all $r \le  n$, we obtain $\gamma(E) > \Cl_n(C)$, a contradiction.
\end{proof}

\begin{rem}\label{rmcs}{\em In the case $n=2$, the hypotheses of Proposition \ref{prop7.1} hold. For $n=3$, they reduce to $\Cl_3(C)\le\Cl_2(C)$. We have seen in this paper that this hypothesis does not always apply.} \end{rem}

\begin{rem}\label{rmcs2}{\em Under the same hypotheses as those of Proposition \ref{prop7.1}, it was proved in \cite{cl4} that $E$ is generated. We therefore have an evaluation sequence
\begin{equation}\label{eval}
0\to M_E\to H^0(E)\otimes\cO_C\to E\to0.
\end{equation}
A version of a conjecture of D. C. Butler \cite{b} states that, for general stable $E$, the kernel $M_E$ should be stable. Of course our bundles are not general, but it is still of interest to ask whether $M_E$ is stable (or semistable) when the hypotheses of Proposition \ref{prop7.1} hold. It has recently been noted by L. Brambila-Paz that the conclusion of the proposition is a necessary condition for the stability of $M_E$. For a line bundle $L$ on a non-hyperelliptic curve $C$, it follows from \cite[Theorem 1.3]{bo} that $M_L$ is stable (this has also been proved by E. Mistretta and L. Stoppino \cite[Corollary 5.5]{ms}).}
\end{rem}
 
\section{Further comments and open problems}\label{further}

There are several problems in connection with Section \ref{minimal}.

\begin{ques} {\em For curves of genus $g\ge14$ satisfying \eqref{eq:new1}, is it true that $\frac{d_9}3-2>\Cl_2(C)$?}
\end{ques}
\noindent{\bf Comment.} Note that by Lemma \ref{lem2.6} the inequality holds for $14\le g\le24$. If the answer to the question is yes, then Theorem \ref{thm2.9} holds for $g\ge16$. The cases $g=14$ and $g=15$ require further investigation.

\begin{ques}{\em Can we extend Theorem \ref{thm2.9} to values of $g$ below $16$?}
\end{ques}

\begin{ques}{\em On curves satisfying \eqref{eq:new1}, can we determine $\Cl_3(C)$ and identify bundles computing it? If so, do any of these bundles fail to be generated?}
\end{ques}
\noindent{\bf Comment.} In connection with the last question, see Proposition \ref{prop7.1} and Remark \ref{rmcs2}.

\medskip
Moving on to Section \ref{plane}, the following question looks interesting.

\begin{ques}{\em For the hyperplane bundle $H$ on a (smooth) plane curve, is it true that $h^0(E_H\otimes E_H)\ge10$?}
\end{ques}
\noindent{\bf Comment.} It seems possible that the answer to this question is known. Note that for a smooth plane curve of degree $\delta\ge7$, we have $\Cl_3(C)=\frac{2\delta-6}3$ if and only if the answer is yes.

\begin{ques}{\em If $C$ is the normalisation of a nodal plane curve $\Gamma$, under what conditions is it true that $d_4\ge2\delta-4$?}
\end{ques}

\begin{ques}{\em For a curve $C$ as in the previous question, under what conditions is it true that $C$ possesses a unique line bundle $H$ of degree $\delta$ with $h^0(H)=3$?}
\end{ques}

\noindent{\bf Comment.} We know this is true if $\nu=0$. It is also true whenever every pencil on $C$ is represented as a pencil of lines through a point of $\Gamma$. Under more restrictive conditions on $\nu$ than that given by \eqref{eqnu}, but without any assumptions of general position and allowing simple cusps as well as nodes, this is shown to be true in \cite[Theorems 2.4 and 5.2]{ck}.

\medskip
Turning to Section \ref{three}, we can ask

\begin{ques}{\em For curves of Clifford index $3$, can we determine when $\Cl_3(C)=\frac83$?}
\end{ques}
\noindent{\bf Comment.} Any such curve must be one of the following:
\begin{itemize}
\item a smooth plane septic;
\item a curve of genus $g$, $7\le g\le 14$, which is representable by a singular plane septic;
\item a curve of genus $8$ with $d_2=8$.
\end{itemize}

\end{document}